\documentclass{amsart}
\usepackage{amsmath,amssymb,amsfonts}
\usepackage{pdfpages}
\usepackage{listings}
\newtheorem{theorem}{Theorem}

\usepackage{graphicx}
\begin{document}

\title{Dimitrov's question for the polynomials of degree 1,2,3,4,5,6}

\author{Dmitriy Dmitrishin, Ivan Skrinnik, Andrey Smorodin, and Alex Stokolos}


\maketitle

\begin{abstract}
D.~Dimitrov \cite{D} has posted the problem of finding the optimal polynomials that provide the sharpness of Koebe Quarter Theorem for polynomials and  asked whether Suffridge polynomials \cite{S} are optimal ones. We disproved Dimitrov's conjecture for polynomials of degree 3,4,5 and 6. For polynomials of degree 1 and 2 the conjecture is valid.

\end{abstract}

\section{Introduction}  
One of the fundamental results in the geometric complex analysis is the famous Koebe Quarter Theorem.  It states that for any function $f\in \mathcal U_n$ the image $f(\mathbb D)$ contains a disc of radius 1/4, whether $\mathbb D=\{|z|<1\}$ and ${\mathcal U_n}=\{f(z): f(0)=0, f^\prime(0)=1,\; f(z) \;\mbox{is univalent in}\;\mathbb D\}.$ The 1/4 bound is sharp as it is indicated by the Koebe function $K(z)=z/(1-z)^2.$
A natural question is whether the constant 1/4 can be improved for  polynomial of specific degree. Say, for polynomials of the first degree it is trivially 1; a simple computation demonstrates that for polynomials of degree 2 it is 1/2. The task was formalized by Dimitrov \cite[Problem 5]{D} who posted the following problem

{\it 
 For any $n\in\mathbb Z_+$ find the polynomial $p_n(z)\in \mathcal U_n$, for which the infimum
$
\inf\{|p_n(z)| : z = e^{it}, 0\le t\le2\pi\}
$
is attained.}

By the Koebe Quarter Theorem the above infimums are bounded from below by 1/4.

C\'ordova and Ruscheweyh \cite{CR} considered the Suffridge polynomials \cite{S}
$$
S_{n,j}(z) = \sum^n_{k=1}
\left(1 -\frac{k-1}n\right) \frac{\sin(\pi jk /(n + 1))}{\sin(\pi j /(n + 1))}z^k.
$$
Note that $S_{n,j}(z)\in \mathcal U_n$ and $|S_{n,1}(-1)|=\frac14\frac{n+1}n\sec^2\frac\pi{2(n+1)}\to −1/4.$ Hence these polynomials solve the latter problem at least asymptotically.  

Note that the value $\frac14\frac{n+1}n\sec^2\frac\pi{2(n+1)}$ is the Koebe radius only for polynomials $S_{n,1}(z)$ of even degree. For the polynomials of odd degree the quantity $\inf\{|S_{n,1}(z)|:|z|=1\}$ is not achived at the point $z=-1,$ rather a different point $\xi,$ such that $S^\prime_{n,1}(\xi)=0.$ (see Fig 1).

\centerline{
\includegraphics[scale=0.25]{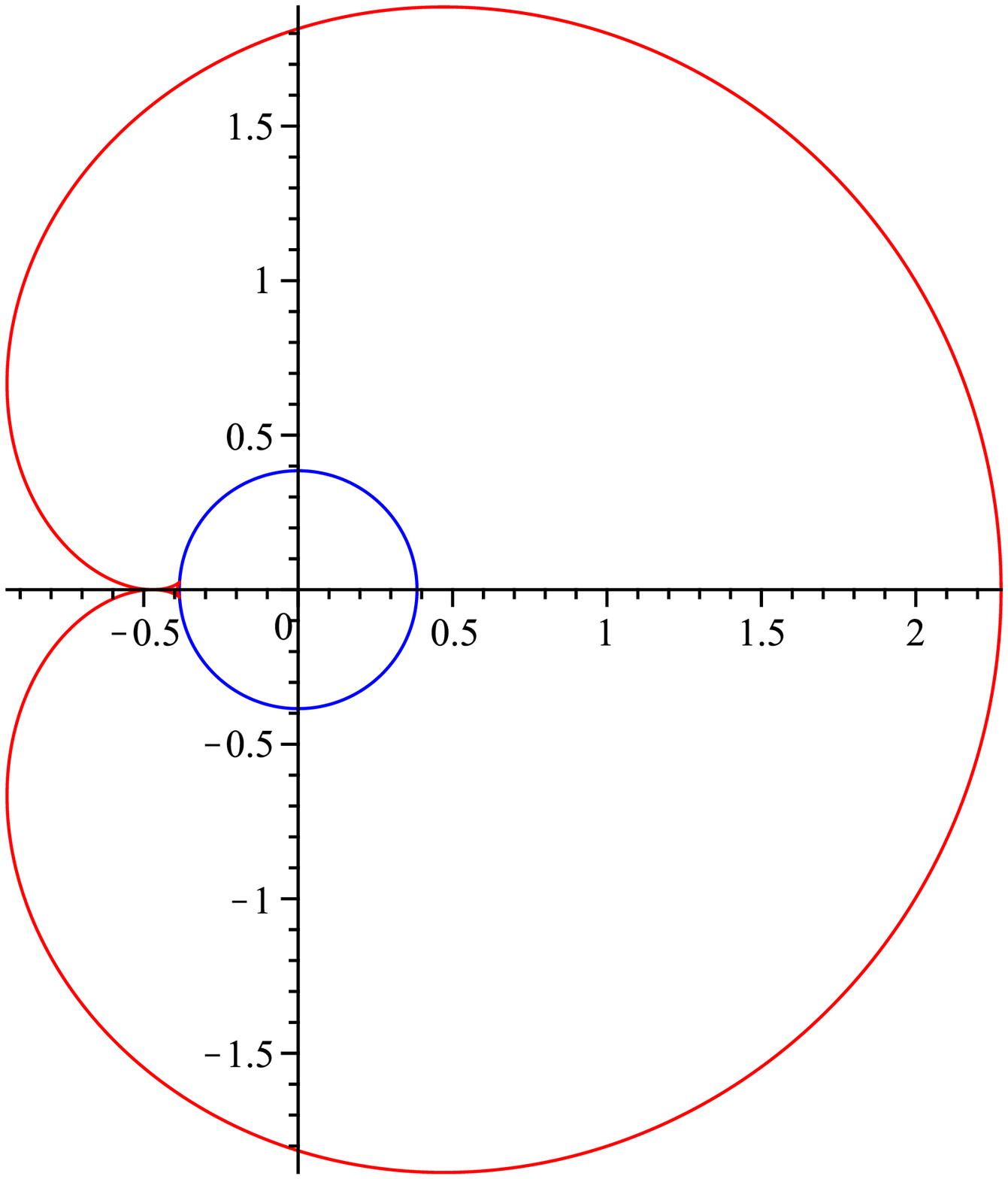}
\includegraphics[scale=0.25]{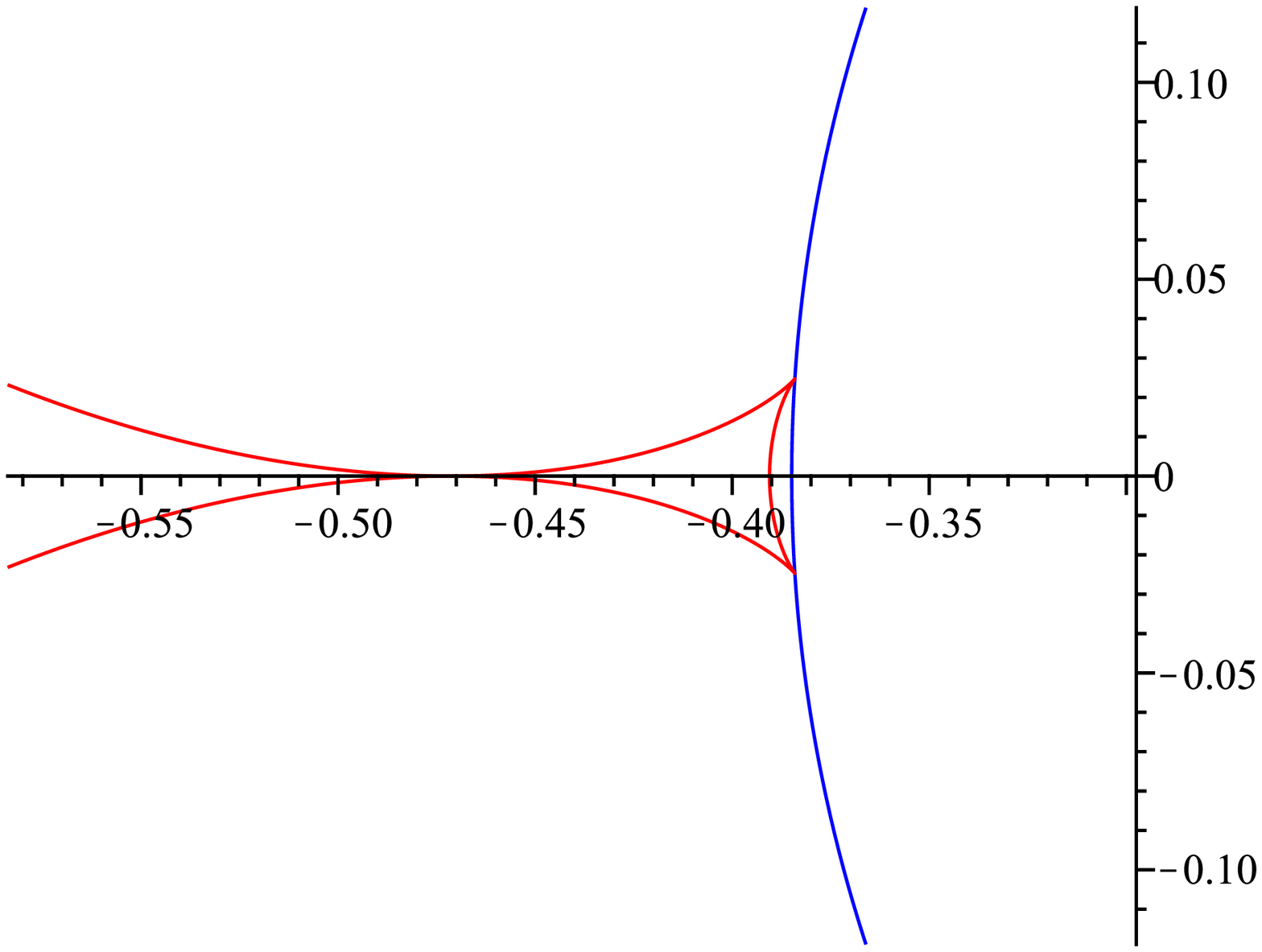}}
\centerline{Fig 1: The image and fragment for $S_{3,1}(\mathbb D).$}
\medskip

Note that for $n=3,$ $|S_{3,1}(-1)|\approx0.3905$ while the Koebe radius is $r_3\approx 0.3849.$ For $n=5$ $|S_{5,1}(-1)|\approx0.3215$ while the Koebe radius is $r_5\approx 0.3196.$ Note that  $r_2=0.5,$ $r_4\approx 0.3455,$ $r_6\approx 0.3069.$

Dimitrov \cite[p.15]{D} asked a specific question about the Suffridge polynomial: {\it  Is it the extremal one for every fixed $N$}? Note that they are indeed extremal for $N=1,2.$ Below we prove that the answer is negative for $N=3,4,5,6.$

\section{New extremal polynomials}

Univalent polynomials are classical objects of complex analysis. Perhaps, the first systematic approach was suggested by Alexander \cite{A} who proved that the truncated sums of the Taylor series of the function $f(z)=\log(1/(1-z))$ are univalent in $\mathbb D$ polynomials. Note that Alexander's paper contains many ideas that were not properly estimated at that time , c.f. \cite{GH}. The subtlety of the situation well illustrates the fact that a necessary condition of univalency - the derivative does not vanish in $\mathbb D$ - implies that the n-th coefficient of the polynomial of degree $n$ cannot exceed $1/n$ in absolute value. This is perfectly fine with the logarithm function and awfully wrong with the Koebe function. Thus, Suffridge polynomials can be treated as reasonable substitutions for the function $K(z).$ These polynomials are extremal in a way that they have the $n$-th coefficient exactly $1/n.$

Thus, so far we have two families of extremal univalent polynomials in play - Alexander polynomials and Suffridge polynomials. The main discovery of the current paper is a new extremal family of polynomials that seem to be univalent in $\mathbb D$ and might be as important as the two mentioned above series.
Namely, the following polynomials were introduced in \cite{DSS}.

\begin{equation*}
P_N(z)=
 \frac1{ { { U'_N}\left ({\cos \frac{\pi } 
{ N+2 } } \right) } }\sum_{k=1}^N
 { U'_{ N - k + 1 } }
\left (\cos \frac{\pi } { N+2 } \right) 
{ U_ { k - 1 } }\left (\cos \frac{\pi }{ N+2 } \right)z^k,  
\end{equation*}
where \({ U_k}\left(x \right) \) is a family of Chebyshev polynomials of the second kind and \({ U'_k}\left(x \right) \) is a derivative.

One given below some examples: 
$$
P_1(z)=z,\quad
P_2(z)=z+\frac12z^2,\quad
$$
$$
P_3(z)=z+\frac2{\sqrt5}z^2+\frac12\left(1- \frac1{\sqrt5}\right)z^3,\quad P_4(z)=z+\frac76z^2+\frac23z^3+\frac16z^4,
$$

\begin{align*}
P_5(z)=
z+&{\frac {8-40\, \left( \cos \left( \pi/7 \right)  \right) ^{2}+32\,
 \left( \cos \left( \pi/7 \right)  \right) ^{3}-24\,\cos \left( \pi/7
 \right) }{40\, \left( \cos \left( \pi/7 \right)  \right) ^{3}-30\,
\cos \left( \pi/7 \right) -32\, \left( \cos \left( \pi/7 \right) 
 \right) ^{2}+7}}{z}^{2}+ \\
&{\frac {24\, \left( \cos \left( \pi/7
 \right)  \right) ^{3}-28\, \left( \cos \left( \pi/7 \right)  \right) 
^{2}-18\,\cos \left( \pi/7 \right) +4}{40\, \left( \cos \left( \pi/7
 \right)  \right) ^{3}-30\,\cos \left( \pi/7 \right) -32\, \left( \cos
 \left( \pi/7 \right)  \right) ^{2}+7}}{z}^{3}+\\
&{\frac {16\, \left( 
\cos \left( \pi/7 \right)  \right) ^{3}-16\, \left( \cos \left( \pi/7
 \right)  \right) ^{2}-12\,\cos \left( \pi/7 \right) +4}{40\, \left( 
\cos \left( \pi/7 \right)  \right) ^{3}-30\,\cos \left( \pi/7 \right) 
-32\, \left( \cos \left( \pi/7 \right)  \right) ^{2}+7}}{z}^{4}+\\
&{\frac {8\, \left( \cos \left( \pi/7 \right)  \right) ^{3}-4\, \left( 
\cos \left( \pi/7 \right)  \right) ^{2}-6\,\cos \left( \pi/7 \right) +
1}{40\, \left( \cos \left( \pi/7 \right)  \right) ^{3}-30\,\cos
 \left( \pi/7 \right) -32\, \left( \cos \left( \pi/7 \right)  \right) 
^{2}+7}}{z}^{5}
\end{align*}

$$
P_6(z)=
z+{\frac {9+8\,\sqrt {2}}{4\,\sqrt {2}+8}}{z}^{2}+{\frac {6\,\sqrt {2
}+10}{4\,\sqrt {2}+8}}{z}^{3}+{\frac {4\,\sqrt {2}+6}{4\,\sqrt {2}+8}}
{z}^{4}+{\frac {2\,\sqrt {2}+2}{4\,\sqrt {2}+8}}{z}^{5}+ \frac1{4\,
\sqrt {2}+8}{z}^{6}
$$

\begin{theorem}   
The following presentation is valid for $t\in(0,\pi),\; t\not=\frac{2\pi}{N+2}$ 
$$
\notag
P_N(e^{it})=\frac{1}{2\left( \cos t- \cos\frac{2\pi}{N+2}\right)}+\frac{1-\cos\frac{2\pi}{N+2}}{(N+2)(1-\cos t)}
\frac{\sin t\sin\frac{N+2}{2}t}{\left(\cos t-\cos\frac{2\pi}{N+2}\right)^2} e^{\frac{N+2}{2}it}.
$$
\end{theorem}

\begin{proof}
First, let us write $P_N(z)$ in terms of trigonometric expressions \cite{DSS}
\begin{align*}
P_N(z)=
 \frac1{ \left(N+2\right)\sin\frac{2\pi}{N+2}}\sum_{k=1}^N\bigg[&\left(N-k+3\right)\sin\frac{\left(k+1\right)\pi}{N+2}-\\
&\left(N-k+1\right)\sin\frac{\left(k-1\right)\pi}{N+2}\bigg]\sin\frac{k\pi}{N+2}z^k
\end{align*}

Having in mind that
$$
\notag
\bigg[2\sin(\pi)-0\cdot \sin\frac{N\pi}{N+2}\bigg]\sin\frac{\left(N+1\right)\pi}{N+2}z^{N+1} \equiv 0
$$
we can change the upper bound for the range in the sum from $N$ to $N+1$. Further modification produces
$$
\notag
P_N(z)=
 \frac1{ \left(N+2\right)\sin\frac{2\pi}{N+2}}\sum_{k=1}^{N+1}\bigg[\left(N-k+2\right)\sin\frac{2k\pi}{N+2}+
2\cdot\frac{\cos\frac{\pi}{N+2}}{\sin\frac{\pi}{N+2}}\sin^2\frac{k\pi}{N+2}\bigg]z^k.
$$
An important observation is that
$$
\notag
\frac{N+1}{N+2}\cdot S_{N+1,2}(z)=\frac1{ \left(N+2\right)\sin\frac{2\pi}{N+2}}\sum_{k=1}^{N+1}\left(N-k+2\right)\sin\frac{2k\pi}{N+2}\cdot{z^k},
$$
where $S_{N+1,2}(z)$ is the second Suffridge polynom of order $N+1$. By formula (5) in \cite[p. 496]{S}, for $n=N+1$ and $j=2$ we get

\begin{align*}
\frac{N+2}{N+1}\cdot S_{N+1,2}\left(e^{it}\right)=\frac{1}{2\left(\cos t - \cos\frac{2\pi}{N+2}\right)}+
\frac{1}{N+2}\cdot \frac{\sin t \cdot \sin\frac{N+2}{2}t}{\left(\cos t-\cos\frac{2\pi}{N+2}\right)^2}\cdot e^{\frac{N+2}{2}it}
\end{align*}
Meanwhile
$$
\sum_{k=1}^{N+1}\sin^2\frac{k\pi}{N+2}e^{ikt}=\sin^2\frac{\pi}{N+2}\cdot\frac{\sin\frac{N+2}{2}t}{\cos{t}-\cos\frac{2\pi}{N+2}}\cdot
\frac{\sin t}{1-\cos t}\cdot e^{i\frac{N+2}{2}t}
$$
By combining both formulas, we get the formula in the theorem.
\end{proof}
Note that the right hand side has removable singularities, thus in fact it is a trigonometric polynomial.\\

 Let us fix a positive integer $N$ and let $R_N(e^{it})=|P_N(e^{it})|^2.$  The following theorem can be directly verified by tedious standard computations.

\begin{theorem}\label{thm2}   
The following presentation is valid for $t\in(0,\pi),\; t\not=\frac{2\pi}{N+2}$:
\begin{align*}
4|P_N(e^{it})|^2=&\left(\frac{\cos\frac{N+2}{2}t}{\cos t- \cos\frac{2\pi}{N+2}}+\frac2{N+2}\frac{1-\cos\frac{2\pi}{N+2}}{1-\cos t}
\frac{\sin t}{(\cos t- \cos\frac{2\pi}{N+2})^2}\sin\frac{N+2}{2}t\right)^2+\\
&\left(\frac{\sin\frac{N+2}{2}t}{\cos t- \cos\frac{2\pi}{N+2}} \right)^2.
\end{align*}

\end{theorem}

Because the real coefficients symmetry of $P_N(e^{it})$ (the real part is an even function and the imaginary is an odd function of $t$), we denote $|P_N(e^{it})|^2=R_N(x)$ as a polynomial of $x=\cos(t).$ Let $b=\cos\frac{2\pi}{N+2}$ and $T_N$ be the Chebyshev polynomial of the first kind. 
From Theorem \ref{thm2} one can get the following formulas by straightforward computations:

$$
4R_N(x)=
\frac1{\left( x-b  \right) ^{2}}+
2\,{\frac { \left( 1-b\right)  \left( 1+x \right) U_{N+1} \left(x
 \right) }
 { \left( N+2 \right)  
 \left( x-b  \right) ^{3}}}+
2\,{\frac { \left( 1-b  \right) ^{2} \left( 1+x
 \right)  \left( 1- T_{N+2} \left( x \right)  \right) }{ \left( N+
2 \right) ^{2} \left( x-b  \right) ^{4} \left( 1-x
 \right) }},
$$

\begin{align*}
4(R_N(x))^\prime=\frac2{\left( b-x  \right) ^{3}}\left(1-\frac{1-b}{1-x}\left(1-T_{N+2}(x)\right)\left(1-\frac{4(1-b)(1+x)}{(N+2)^2(b-x)^2}-\frac{2(1-b)}{(N+2)^2(1-x)(b-x)}\right)+\right.\\
\frac{1-b}{1-x}-\frac{1-b}{1-x}\frac{1}{N+2}U_{N+1}(x)\frac{1-bx+3(1-x^2)}{b-x}\bigg). 
\end{align*}

\begin{theorem}\label{thm3}  
If $(R_N(x))^\prime>0$ for $x\in(-1,1)$ then the polynomial $P_N(z)$ is univalent in $\mathbb D$ and the Koebe radius of this polynomial is $\sqrt{R_N(-1)}.$
\end{theorem}  

It is proved in \cite{DSS} that the polynomial $P_N(z)$ is typically real and thus the image of the unit circle has no self intersections, the theorem is proved.

Note, that \cite{DSS}
$$  
\sqrt{R_N(-1)}=\frac14\sec^2\frac\pi{N+2}.
$$

\section{The case N=1.}
In this case $R_1(x)=1,$ thus the Koebe radius is 1. 

\section{The case N=2.}
In this case $R_2(x)=5/4+x,$ thus the Koebe radius is $\sqrt{R_2(-1)}=1/2.$

\section{The case N=3.}
In this case the polynomial $P_3(z)$ is univalent that can be verified using Brennan's criteria \cite{B}. Also

\begin{align*}
4R_3(x)=
-&\frac {2}{25}\,{\frac {37\,\cos \left( 1/5\,\pi  \right) -69+56\,
 \left( \cos \left( 1/5\,\pi  \right)  \right) ^{2}}{ \left( \cos
 \left( 1/5\,\pi  \right)  \right) ^{2}+1-2\,\cos \left( 1/5\,\pi 
 \right) }}-{\frac {32}{25}}\,{\frac { \left( 23\,\cos \left( 1/5\,
\pi  \right) -51+49\, \left( \cos \left( 1/5\,\pi  \right)  \right) ^{
2} \right) x} {-9-5\,\cos \left( 1/5\,\pi  \right) +
20\, \left( \cos
 \left( 1/5\,\pi  \right)  \right) ^{2}}}-\\
 &{\frac {32}{5}}\,{\frac {
 \left( 10\,\cos \left( 1/5\,\pi  \right) -14+9\, \left( \cos \left( 1
/5\,\pi  \right)  \right) ^{2} \right) {x}^{2}}{14\, \left( \cos
 \left( 1/5\,\pi  \right)  \right) ^{2}+3-15\,\cos \left( 1/5\,\pi 
 \right) }}
\end{align*}

$$
4R_3^\prime(x)=-{\frac {32}{25}}\,{\frac {23\,\cos \left( 1/5\,\pi  \right) -51+49\,
 \left( \cos \left( 1/5\,\pi  \right)  \right) ^{2}}{-9-5\,\cos
 \left( 1/5\,\pi  \right) +20\, \left( \cos \left( 1/5\,\pi  \right) 
 \right) ^{2}}}-{\frac {64}{5}}\,{\frac { \left( 10\,\cos \left( 1/5\,
\pi  \right) -14+9\, \left( \cos \left( 1/5\,\pi  \right)  \right) ^{2
} \right) x}{14\, \left( \cos \left( 1/5\,\pi  \right)  \right) ^{2}+3
-15\,\cos \left( 1/5\,\pi  \right) }}
$$

One can check that $R_3^\prime(x)$ is positive on [-1,1], which implies the estimate from above on Koebe radius $|P_3(-1)|=\frac{3-\sqrt 5}2\approx 0.382.$
 
\section{The case N=4.}
In this case the polynomial $P_4(z)$ is univalent, c.f. \cite{J}.
$$
4R_4(x)=40/9+(112/9)x+(124/9)x^2+(16/3)x^3
$$
and
$$
4R_4^\prime(x)=112/9+(248/9)x+(48/3)x^2
$$
The discriminant is $-37.13...$ therefore the smallest value for $R_4(x)$ is at -1,
 which implies the estimate from above on Koebe radius $|P_4(-1)|=1/3.$
 
\section{The case N=5.}

In the particular case $N=5$ we get

\begin{align*}
4R_5^\prime(x)=
&{\frac {16}{49}}\,{\frac {42\, \left( \cos \left( 1/7\,\pi  \right) 
 \right) ^{3}-31\,\cos \left( 1/7\,\pi  \right) +9-47\, \left( \cos
 \left( 1/7\,\pi  \right)  \right) ^{2}}{2\, \left( \cos \left( 1/7\,
\pi  \right)  \right) ^{3}+\cos \left( 1/7\,\pi  \right) -10\, \left( 
\cos \left( 1/7\,\pi  \right)  \right) ^{2}+5}}+\\
&{\frac {64}{49}}\,{
\frac { \left( 762\, \left( \cos \left( 1/7\,\pi  \right)  \right) ^{3
}-618\,\cos \left( 1/7\,\pi  \right) -3-323\, \left( \sin \left( 1/7\,
\pi  \right)  \right) ^{2} \right) x}{-11+2\, \left( \sin \left( 1/7\,
\pi  \right)  \right) ^{2}-39\,\cos \left( 1/7\,\pi  \right) +60\,
 \left( \cos \left( 1/7\,\pi  \right)  \right) ^{3}}}+\\
&{\frac {192}{49}
}\,{\frac { \left( 940\, \left( \cos \left( 1/7\,\pi  \right) 
 \right) ^{3}-761\,\cos \left( 1/7\,\pi  \right) +81-536\, \left( \sin
 \left( 1/7\,\pi  \right)  \right) ^{2} \right) {x}^{2}}{-18+14\,
 \left( \cos \left( 1/7\,\pi  \right)  \right) ^{3}+35\, \left( \sin
 \left( 1/7\,\pi  \right)  \right) ^{2}}}+\\
&{\frac {128}{7}}\,{\frac {
 \left( 380\, \left( \cos \left( 1/7\,\pi  \right)  \right) ^{3}-293\,
\cos \left( 1/7\,\pi  \right) +17-176\, \left( \sin \left( 1/7\,\pi 
 \right)  \right) ^{2} \right) {x}^{3}}{-22+9\, \left( \sin \left( 1/7
\,\pi  \right)  \right) ^{2}-70\,\cos \left( 1/7\,\pi  \right) +112\,
 \left( \cos \left( 1/7\,\pi  \right)  \right) ^{3}}}.
\end{align*}

\centerline{
\includegraphics[scale=0.25]{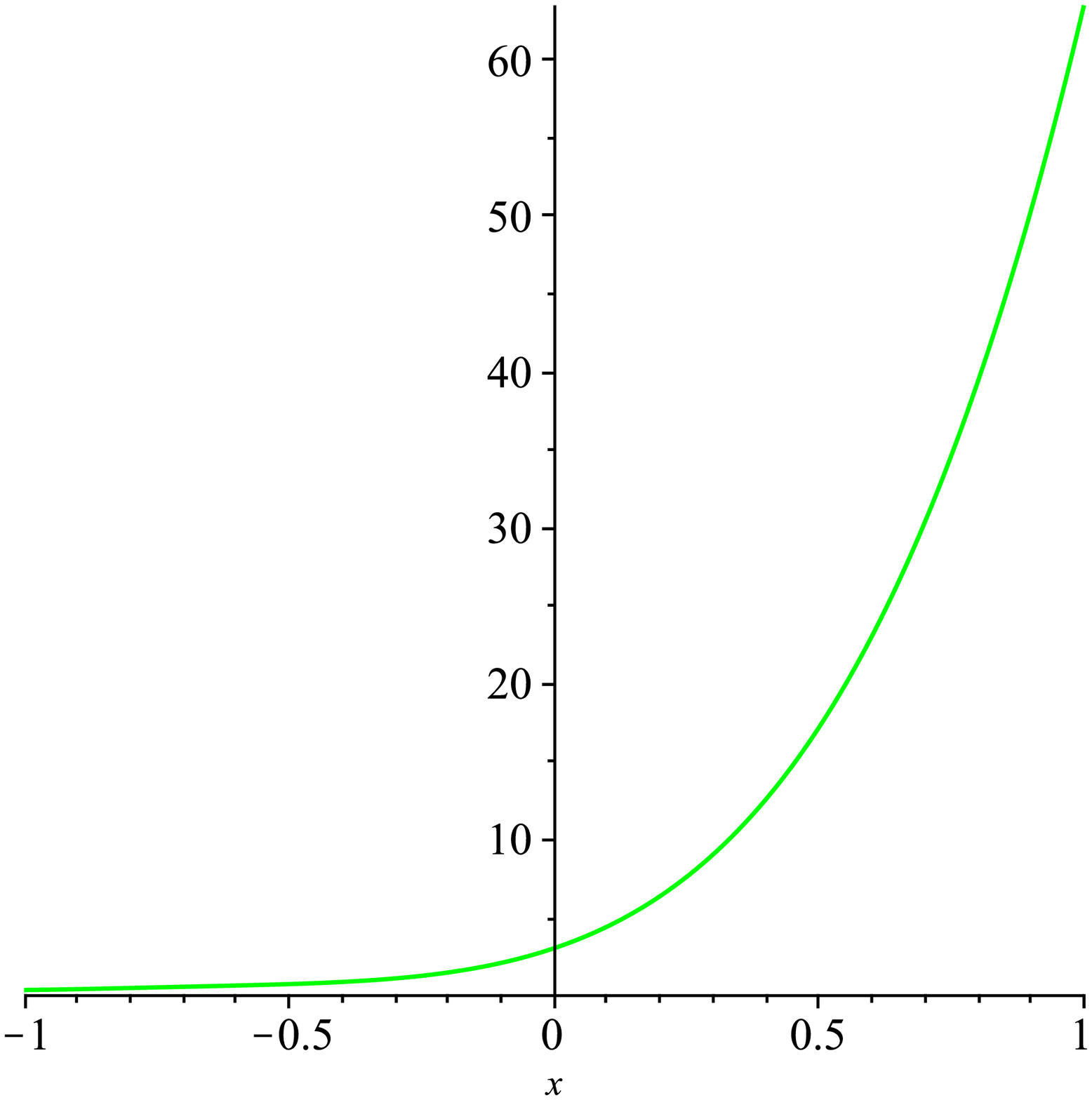}
\includegraphics[scale=0.25]{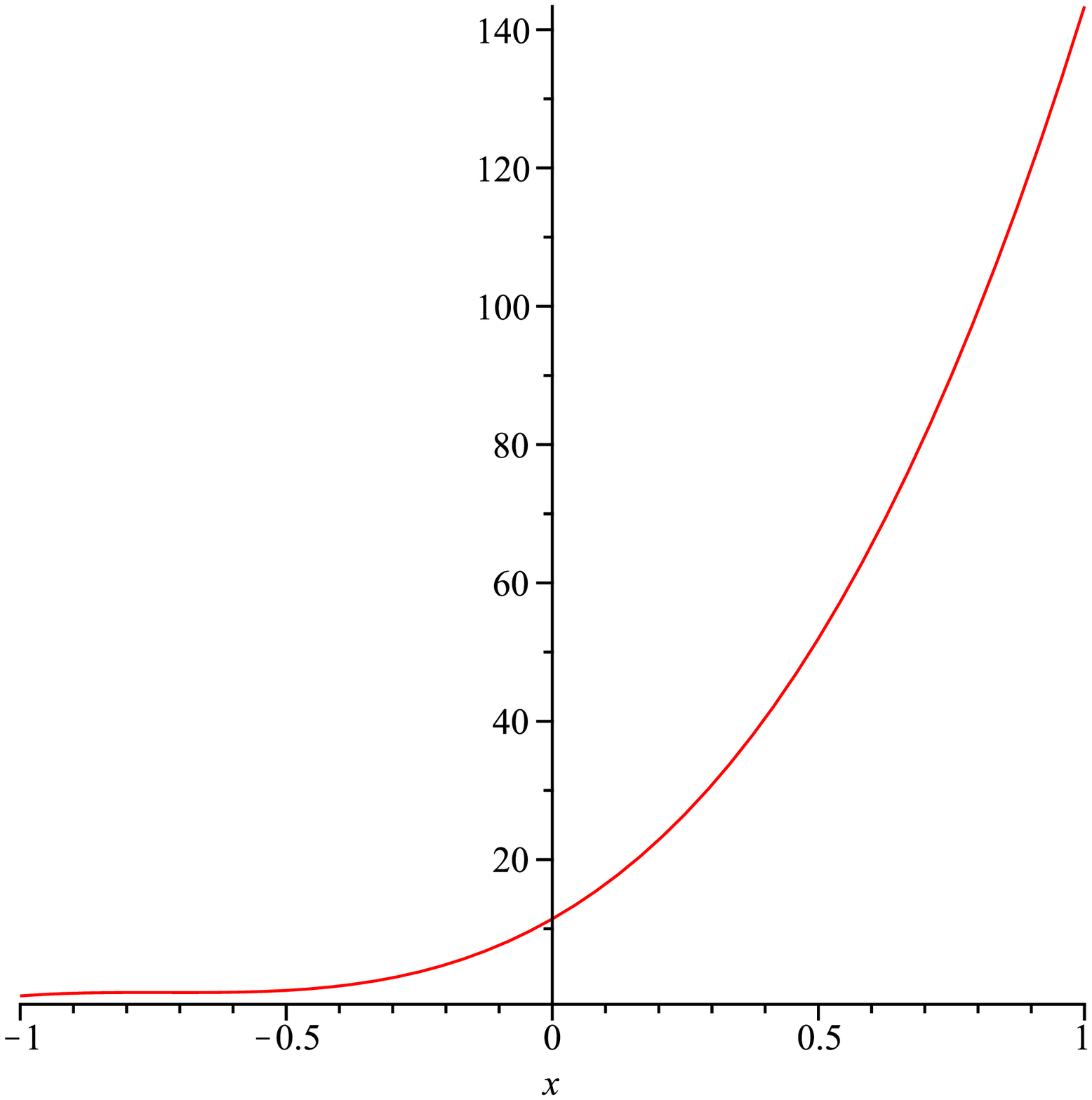}}
\centerline{Fig 2: The graphs $R_5(x)$ and $R^\prime_5(x)$}
\medskip

By decomposing into Taylor polynomial centered at -1 we get

\begin{align*}
4R_5^\prime(x)=
-&{\frac {8}{49}}\,{\frac {778868087\, \left( \cos \left(\frac{1}{7}\,\pi 
 \right)  \right) ^{2}-791395834+2270258054\, \left( \cos \left(\frac{1}{7}\,
\pi  \right)  \right) ^{3}-1666223113\,\cos \left(\frac{1}{7}\,\pi  \right) }
{10381281\, \left( \cos \left(\frac{1}{7}\,\pi  \right)  \right) ^{3}-5460108
\,\cos \left(\frac{1}{7}\,\pi  \right) -4219605\, \left( \cos \left(\frac{1}{7}\,
\pi  \right)  \right) ^{2}+752170}}\\
-&{\frac {32}{49}}\,{\frac { \left( 
-58704325+88578183\, \left( \cos \left(\frac{1}{7}\,\pi  \right)  \right) ^{2
}+57393568\, \left( \cos \left(\frac{1}{7}\,\pi  \right)  \right) ^{3}-
61237982\,\cos \left(\frac{1}{7}\,\pi  \right)  \right)  \left( 1+x \right) }
{-360031\, \left( \cos \left(\frac{1}{7}\,\pi  \right)  \right) ^{2}+121875+
515550\, \left( \cos \left(\frac{1}{7}\,\pi  \right)  \right) ^{3}-229418\,
\cos \left(\frac{1}{7}\,\pi  \right) }}+\\
&{\frac {192}{49}}\,{\frac { \left( -
212691+312494\, \left( \cos \left(\frac{1}{7}\,\pi  \right)  \right) ^{2}-
238981\,\cos \left(\frac{1}{7}\,\pi  \right) +238432\, \left( \cos \left(\frac{1}{7}
\,\pi  \right)  \right) ^{3} \right)  \left( 1+x \right) ^{2}}{2543-
6451\, \left( \cos \left(\frac{1}{7}\,\pi  \right)  \right) ^{2}-2632\,\cos
 \left(\frac{1}{7}\,\pi  \right) +6916\, \left( \cos \left(\frac{1}{7}\,\pi 
 \right)  \right) ^{3}}}+\\
&{\frac {128}{7}}\,{\frac { \left( 17-176\,
 \left( \sin \left( \frac{1}{7}\,\pi  \right)  \right) ^{2}+380\, \left( \cos
 \left(\frac{1}{7}\,\pi  \right)  \right) ^{3}-293\,\cos \left(\frac{1}{7}\,\pi 
 \right)  \right)  \left( 1+x \right) ^{3}}{-22+9\, \left( \sin
 \left(\frac{1}{7}\,\pi  \right)  \right) ^{2}-70\,\cos \left(\frac{1}{7}\,\pi 
 \right) +112\, \left( \cos \left(\frac{1}{7}\,\pi  \right)  \right) ^{3}}}.
\end{align*}

Thus, $R_5^\prime(x)=A_0+(x+1)(A_1+A_2(x+1)+A_3(x+1)^2)$ with the obviuos choice of $A_j.$ Since for $|x|\le1$ the value $x+1$ is positive and $A_i\ge0$ for $i=0,1$ then the inequality
\begin{equation}\label{discr}
A_2^2-4A_1A_3<0
\end{equation}
implies that 
\begin{equation}\label{der}
R^\prime(x)> 0;\quad x\in[-1,1].
\end{equation}
The verification of \eqref{discr} is an elementary issue based on approximations of  $\cos{\pi/7}$ and $\sin{\pi/7}$ from above and below with sufficiently large number of digits. 

This proves that the derivative does not intersect the interval and that $R^\prime_5(z)\geq{0}$. Thus, $R_5(z)$ is not decreasing on $\left[-1,1\right]$ therefore  $P_5(z)$ is univalent by Theorem \ref{thm3}. This gives us an estimate on the Koebe radius $|P_5(-1)|\approx0.3080.$ 

\section{The case $N=6.$} 
In this case

\begin{eqnarray*}
4R_6(x) &=&2+\left( 8\,\sqrt {2}-4 \right) x+ \left( 38-12\,\sqrt {2} \right) {x}^{2}+ \left( 28\,\sqrt {2}-4 \right) {x}^{3}+\\
& & \left( 28\,\sqrt {2}-10\right) {x}^{4}+\left( -16\,\sqrt {2}+32 \right) {x}^{5}.
\end{eqnarray*}

\begin{eqnarray*}
4R^\prime_6(x)&=&8\,\sqrt {2}-4+2\, \left( 38-12\,\sqrt {2} \right) x+3\, \left( 28\,\sqrt {2}-4 \right) {x}^{2}+\\
& & 4\, \left( 28\,\sqrt {2}-10 \right) {x}^{3}+5\, \left( -16\,\sqrt {2}+32 \right) {x}^{4}\\
&= & -76\,\sqrt {2}+108+ \left( 464\,\sqrt {2}-660 \right)  \left( 1+x\right) + \left( -732\,\sqrt {2}+1068 \right)  \left( 1+x \right) ^{2}+\\
& & \left( 432\,\sqrt {2}-680 \right)  \left( 1+x \right) ^{3}+ \left( -80\,\sqrt {2}+160 \right)  \left( 1+x \right) ^{4}\\
&=&\left( 108-76\, \sqrt{2}+\frac{ -660+464\, \sqrt{2}}{2\, \sqrt{108-76\, \sqrt{2}}} \left( x+1 \right)  \right) ^{2}+
\left[{\frac { \left( -660+464\, \sqrt{2} \right) ^{2} }{4 \left( 108-76\, \sqrt{2} \right) }}-732\, \sqrt{2}+1068 \right.+\\
& & \left. \left( 432\, \sqrt{2}-680\right)  \left( x+1 \right)+\left( -80\, \sqrt{2}+160 \right)\left( x+1 \right) ^{2}\right]  \left( x+1 \right) ^{2}
\end{eqnarray*}

Applying on argument similar to the formula \eqref{discr} we get formula \eqref{der} which implies the estimate for the Koebe radius $|P_6(-1)|\approx 0.2929.$ We conjecture that the obtained estimates in fact are true values.

\section{Conclusion}
In \cite{DSS} a new class of polynomials was introduced
\begin{equation*}
P_N(z)=
 \frac1{ { { U'_N}\left ({\cos \frac{\pi } 
{ N+2 } } \right) } }\sum_{k=1}^N
 { U'_{ N - k + 1 } }
\left (\cos \frac{\pi } { N+2 } \right) 
{ U_ { k - 1 } }\left (\cos \frac{\pi }{ N+2 } \right)z^k,  
\end{equation*}
and the extremal property of these polynomials was mentioned
$$
\sup_{p_N(z)=z+\sum_{k=2}^Na_kz^k}\min_t
\{\Re (p_N(e^{it})): \Im (p_N(e^{it})=0\}=P_N(-1).
$$
It was conjectured that these polynomials are univalent and solves Dimitrov problem.

In the present article the first conjecture is proved for $N=1,...,6$ thus for those $N$ the estimates from below on the radius Koebe of polynomials from $\mathcal U_N$ are obtained. It is shown that those values are smaller then the corresponding ones for Suffridge polynomials $S_{N,1}(z).$

To prove the case $N>6$ one needs to verify the criteria given by Theorem \ref{thm3}, which is a notrivial task. Currently we are working on this subject.

Also, let us mention that the polynomials $P_N(z), S_{N,1}(z)$ and their generalizations turnes out to be very helpful in the problem of stabilization of of cycles in nonlinear discrete systems \cite{DKST,DKS}.

\end{document}